\documentclass[12pt,a4paper]{article}
\usepackage{amsfonts,amssymb,amsmath,amsthm}
\usepackage[T1]{fontenc}
\usepackage[english]{babel} 
\usepackage[utf8]{inputenc}
\usepackage{a4wide}

\begin{document}

\newtheorem{defn}{Definition}[section]
\newtheorem{lemma}[defn]{Lemma}
\newtheorem{ex}[defn]{Example}
\newtheorem{thm}[defn]{Theorem}
\newtheorem{prop}[defn]{Proposition}
\newtheorem{cor}[defn]{Corollary}
\newtheorem{rem}[defn]{Remark}
\newtheorem{hyp}[defn]{Hypothesis}

\title{\bf Invertibility of matrix type operators of infinite order with exponential off-diagonal decay}

\author{
Stevan~Pilipovi\'c\thanks{Faculty of Sciences, University of Novi Sad, Serbia, stevan.pilipovic@dmi.uns.ac.rs} ,
Bojan~Prangoski\thanks{Faculty of Mechanical Engineering, University Ss. Cyril and Methodius, Skopje, Macedonia, bprangoski@yahoo.com} ,
Milica \v Zigi\'c\thanks{Faculty of Sciences, University of Novi Sad, Serbia, milica.zigic@dmi.uns.ac.rs}
	} 
\date{}
\maketitle


\bigskip

\begin{abstract}
We analyse various exponential off-diagonal decay rates of the elements of infinite matrices and their inverses. It is known that such decay of the elements of an infinite matrix does not imply inverse--closeness, i.e. the inverse, if exists, does not have the same order of decay. We discuss some consequences and extensions of this result.
\end{abstract}



\section{Introduction, motivation and basic notation}

We study the decay rate of the entries in the inverse of an invertible matrix type operators $A=(A_{s,t})_{s,t}$ where the entries $A_{s,t}$ decrease  exponentially as the distance $d(s,t)\rightarrow \infty$. We call this decay exponential off-diagonal decay. First we explain the general framework of the investigation.

Recall that the index set $\Lambda\subseteq \mathbb{R}^d$ is called a lattice if $\Lambda=G\mathbb{Z}^d$, where $G$ is a non-singular matrix. Given a complex, separable Hilbert space $H$ we denote by $\mathcal{L}(H)$ the Banach algebra of all bounded linear operators on $H$. If $A\in \mathcal{L}(H)$ then we can represent $A$ as a matrix (which we still denote by $A$) with respect to any complete orthonormal set. We do not loose on generality if we restrict our investigations to the case $H=\ell^2(\Lambda)$ with the ordinary operator norm  $\|A\|$.

A convolution Banach algebra of sequences of polynomial decay $\mathcal A$ is analysed in \cite{Groh2} as a tool for the analysis of the  corresponding  extension of  the Sj\"ostrand class $M^{\infty,1}(\mathbb{R}^{2d})$ (cf. \cite{JS}). The new symbol class $\widetilde{M}^{\infty, \mathcal A}$ is also characterised by the mean of the Toeplitz  matrices $\mathcal C_{\mathcal A}$ with the diagonal elements determined by the elements of $\mathcal A$.

Consider the class of matrix type operator $A$ satisfying the estimate $$|A_{s,t}|\leq C \omega,\; s,t\in\Lambda,$$
with $\omega=\omega_p =(1+|s-t|)^{-k/2}$, for some $k>0$, (i.e. the class of matrices with polynomial off-diagonal decay) or $\omega=\omega_{se} =e^{k|s-t|^\beta}$ for some $\beta\in(0,1)$, (i.e. the class of matrices with sub-exponential off-diagonal decay). One of the most important properties that these classes share is that they are spectrally invariant. This is explained in \cite{sjaff} and later in \cite{Stro}, \cite{GL} and \cite{Groh2}; here we just recall that spectral invariance (or Wiener type property) means that if $A$ is invertible in $\ell^2$, then its inverse has the same off-diagonal decay order (polynomial or sub-exponential).

Recently, a series of papers \cite{cgnr1,cnr2,CNR,cnr3,cnr4} were related to the matrix type characterisation of various classes of pseudo-differential operators and Fourier integral operators through the matrix representation related to the Gabor wave packets and the almost diagonalisation of $\Psi$DOs and FIOs as was suggested for Sj\"ostrand or Gr\"ochenig--Rzeszotnik class of $\Psi$DOs.

Our results are presented as follows. In the second section, we recall the results for band limited matrices \cite{Demko} and show by a simple example specific cases which are not covered by the results of \cite{Demko}. Our example shows that the class of matrices with super-exponential off-diagonal decay is not spectrally invariant. As it is shown by Jaffard \cite{sjaff}, in the case of exponential off-diagonal decay, the rate of the decay does not stay unchanged for the inverse matrix; see \cite{PS} for results concerning such type of operators in the context of Fr\'echet frames. Therefore, in section 3 we give some precise estimates on the decay rate of the inverse matrix which follow from Jaffard's theorem (cf. \cite{PS} with $d(s,t)=|s-t|$).

Our main result is given in the fourth section where we consider different exponential off-diagonal decay rate assumptions and obtain new sufficient conditions characterising matrix type operators. Finally, we consider invertible matrices of infinite order which satisfy estimates of the form $$|A_{s,t}|\leq C_pe^{-pd(s,t)},\;s,t\in\Lambda, \quad\mbox{with}\quad C_p\leq Ke^{p\varphi(p)},\;p\in\mathbb{N},$$ where $\varphi(p)$ is a strictly increasing function which satisfies certain conditions. In the main result, we provide estimates on the off-diagonal decay rate of the entries of $A^{-1}$.


\section{Band limited matrices and sub-exponential off-diagonal decay matrices}

From the chronological point of view, important results were obtained  in \cite{Demko}, where a special class of $m-$banded matrices was considered. Recall that a matrix $A=(A_{s,t})_{\Lambda\times \Lambda}$ is $m-$banded if $A_{s,t}=0$ for $s,t\in\Lambda$, $|s-t|>m$. It is proved in \cite[Proposition 2.3]{Demko} that for an $m-$banded invertible matrix operator $A\in \mathcal{L}(\ell^2(\Lambda))$ with $(A^{-1}_{s,t})_{\Lambda\times\Lambda}=A^{-1}\in \mathcal{L}(\ell^2(\Lambda))$, the following estimate holds true
\begin{equation}\label{D}
|A^{-1}_{s,t}|\leq Ce^{\frac{1}{m}(\ln (1-\frac{2}{\sqrt{\kappa}+1})|s-t|)},\quad s,t\in \Lambda,
\end{equation}
where $C$ is a certain positive constant depending on $m,A$ and $[a,b]$, $a>0$, where $[a,b]$ is the smallest interval containing the spectrum of $AA^*$, while $\kappa=b/a$.

The following result of Jaffard \cite{sjaff} is more general since the band limitation is not assumed. Denote by $\mathcal{E}_\gamma$, $\gamma>0$, the space of matrices $A=(A_{s,t})_{\Lambda\times\Lambda}$ whose entries satisfy:
\begin{align}\label{*}(\exists C_{A,\gamma}\geq 1)\; |A_{s,t}|\leq C_{A,\gamma}e^{-\gamma d(s,t)},\,\,\, s,t\in\Lambda.\end{align}

\begin{thm}(\cite[Propositions 2]{sjaff})\label{t Jaf}
Let $A:\ell^2(\Lambda)\to\ell^2(\Lambda)$ be an invertible matrix on $\ell^2(\Lambda)$. If $A\in \mathcal{E}_\gamma,$ then $A^{-1}\in \mathcal{E}_{\gamma_1}$ for some $\gamma_1\in (0,\gamma)$.
\end{thm}

The next simple example shows that for any $k\in \mathbb{Z}_+$ one can construct a matrix $A$ that belongs to $\mathcal{E}_\gamma$ for all $\gamma>0$, but the inverse $A^{-1}$ is in $\mathcal{E}_{\gamma}$ only for $\gamma\leq 1/k$. Actually this example is mentioned in \cite[Section 3]{Demko} in another general form (more precisely, a matrix $AA^*$ is considered) as one for which \eqref{D} gives useless estimate (as it was written in \cite{Demko}). Because of simplicity, the explicit estimate of the growth rate and for the sake of further comments concerning the super-exponential growth (see Remark \ref{rem 1} - 2.), we proceed with it.

\begin{ex}\label{ex}
Let $k\in \mathbb{Z}_+.$ Consider the matrix $A=I-\Gamma$, where $I$ is the identity matrix and the elements of  $\Gamma=(\Gamma_{s,t})_{\mathbb{Z}\times\mathbb{Z}}$ are given by $\Gamma_{s,t}=e^{-1/k}$ if $s+1=t,\;s,t\in \mathbb{Z}$ and $\Gamma_{s,t}=0$ otherwise, i.e.
	$$\Gamma=
	\begin{pmatrix}
	\ddots & \vdots & \vdots & \vdots & \vdots &   \\
	\dots  & 0 & e^{-1/k} &0 & 0 &\dots  \\
	\dots  & 0 & 0 & e^{-1/k} & 0 & \dots  \\
	\dots  & 0 & 0 & 0 & e^{-1/k} & \dots  \\
	\dots  & 0 & 0 & 0 & 0 & \dots  \\
	       & \vdots & \vdots & \vdots & \vdots & \ddots
	\end{pmatrix}
	\;\mbox {and}
	$$
	$$ A=I-\Gamma=
	\begin{pmatrix}
	\ddots & \vdots & \vdots & \vdots & \vdots &   \\
	\dots  & 1 & -e^{-1/k} &0 & 0 &\dots  \\
	\dots  & 0 & 1 & -e^{-1/k} & 0 & \dots  \\
	\dots  & 0 & 0 & 1 & -e^{-1/k} & \dots  \\
	\dots  & 0 & 0 & 0 & 1 & \dots  \\
	       & \vdots & \vdots & \vdots & \vdots & \ddots
	\end{pmatrix}
	$$
(in other words, $\Gamma$ is $e^{-1/k}$ times the backward shift on $\ell^2$). Clearly, for every $\gamma>0$ there exists $C_{\gamma}>0$ such that $|A_{s,t}|\leq C_\gamma e^{-\gamma |s-t|}$, $s,t\in \mathbb{Z}$; i.e. $A$ belongs to $\bigcap_{\gamma>0} \mathcal{E}_\gamma$. As $\|\Gamma\|=e^{-1/k}<1$, $A^{-1}$ exists and $A^{-1}=\sum_{n=0}^\infty \Gamma^n$. One easily verifies that the entries in $\Gamma^n=(\Gamma^n_{s,t})_{\mathbb{Z}\times\mathbb{Z}}$ are given by $\Gamma^n_{s,t}=e^{-n/k}$ if $s+n=t$, $n\in\mathbb{Z}_+,\;s,t\in \mathbb{Z}$, and $\Gamma^n_{s,t}=0$ otherwise. Consequently, the entries in the inverse $A^{-1}=(B_{s,t})_{\mathbb{Z}\times\mathbb{Z}}$ are as follows: $B_{s,t}=e^{-(1/k)|s-t|}$ for $s\leq t,\;s,t\in\mathbb{Z}$ and $B_{s,t}=0$ for $s>t,\;s,t\in\mathbb{Z}$; i.e.
	$$A^{-1}=B=
	\begin{pmatrix}
	\ddots & \vdots & \vdots & \vdots & \vdots &   \\
	\dots  & 1 & e^{-1/k} &e^{-2/k} & e^{-3/k} &\dots  \\
	\dots  & 0 & 1 & e^{-1/k} & e^{-2/k} & \dots  \\
	\dots  & 0 & 0 & 1 & e^{-1/k} & \dots  \\
	\dots  & 0 & 0 & 0 & 1 & \dots  \\
	       & \vdots & \vdots & \vdots & \vdots & \ddots
	\end{pmatrix}
	$$
Hence, one obtains
$$|B_{s,t}|=e^{-(1/k)|s-t|},\quad\mbox{for}\quad s\leq t,\;s,t\in\mathbb{Z},$$
so $A^{-1}\in \mathcal{E}_{\gamma}$ only for $\gamma\leq 1/k$.
	
\end{ex}

The most general result for the sub-exponential growth is obtained in \cite{GL}: Assume that $\rho:[0,\infty)\to [0,\infty)$ is a strictly increasing concave and normalised ($\rho(0)=0$) function that satisfies
\begin{align}\label{**}\lim_{\xi\to \infty}\frac{\rho(\xi)}{\xi}=0.\end{align}
A function $u$ is an admissible weight if $u(x)=e^{\rho(|x|)}$, $x\in \mathbb{Z}^d$.
Let $s>d$ and let the weight $v$ be given by $v(x)=u(x)(1+|x|)^s$, $x\in \mathbb{Z}^d$. Then \cite[Corollary 11]{GL} states the following: if there exists $C>0$ such that
$$|A_{s,t}|\leq C v(s-t)^{-1},\quad s,t\in \mathbb{Z}^d,$$
then there is $C_1>0$ such that
$$|A^{-1}_{s,t}|\leq C_1 v(s-t)^{-1},\quad s,t\in\mathbb{Z}^d.$$
This assertion implies the next one which shows that stronger sub-exponen\-tial off-diagonal decay for the matrix $A$ implies the stronger sub-exponential off-diagonal decay for the inverse matrix $A^{-1}$ (cf. Remark \ref{rem 1} - 1.).

\begin{prop}\label{vckltp135}
Let $A=(A_{s,t})_{\mathbb{Z}^d\times\mathbb{Z}^d}\in\mathcal{L}(\ell^2(\mathbb{Z}^d))$ be invertible with inverse $A^{-1}=(A^{-1}_{s,t})_{\mathbb{Z}^d\times\mathbb{Z}^d}$. Assume that
\begin{equation}\label{a}
(\forall k>0)(\exists C_k>0)\quad
|A_{s,t}|\leq C_k e^{-k\rho(|s-t|)},\;s,t\in\mathbb{Z}^d,
\end{equation}
where $\rho$ satisfies the above assumptions. Moreover, assume that for every $\varepsilon>0$ there exists $\tilde{C}_{\varepsilon}>0$ such that
\begin{equation}\label{b}
\tilde C_{\varepsilon}\frac{e^{\varepsilon\rho(\xi)}}{(1+\xi)^{d+1}}\geq 1,\,\,\, \mbox{for all}\,\, \xi\geq 0.\end{equation}
Then
\begin{equation}\label{c}
(\forall k>0)(\exists C'_k>0)\quad|A_{s,t}^{-1}|\leq C'_{k}e^{-k \rho(|s-t|)},\;s,t\in\mathbb{Z}^d.
\end{equation}
\end{prop}

\begin{proof}
Let $k>0$ be arbitrary but fixed. Employing \eqref{a} and \eqref{b}, we infer
$$|A_{s,t}|\leq C_{2k} e^{-2k\rho(|s-t|)}\leq C_{2k}\tilde{C}_k\frac{e^{-k\rho(|s-t|)}}{(1+|s-t|)^{d+1}},\;s,t\in\mathbb{Z}^d.$$
Consequently, we can invoke \cite[Corollary 11]{GL} to deduce the claim in the proposition.
\end{proof}

\begin{rem}
Clearly, the function $\rho(\xi)=\xi^\beta$, $\beta\in (0,1)$, $\xi\geq 0$, satisfies the conditions in Proposition \ref{vckltp135}.
\end{rem}


\section{Jaffard result and additional estimates}

In this section we analyse, in details, the decay of the inverse matrices of infinite order if the original matrices have exponential off-diagonal decay, that is, if $\rho(\xi)/\xi$ does not converge to zero. For example this is the case if $\rho(\xi)=c\xi$, $\xi\geq 0$. Actually, we follow the proof of Jaffard theorem in \cite{sjaff} and give precise estimates. We consider in \cite{PS} the special case $\rho(\xi)=c\xi$, $\xi\geq 0$, and $v(x)=e^{\rho(|x|)}$, $x\in \mathbb{Z}^d.$ We give just a part of the proof in order to give better estimates for the rate of decay of the elements of the inverse matrices.

Let $\Lambda$ be a discrete index set endowed with a distance $d$ which satisfies the following assumption:
\begin{equation}\label{d exp}
m_{\varepsilon}:=\sup_{s\in \Lambda
}\;\sum_{t\in \Lambda}e^{-\varepsilon d(s,t)}<\infty,\,\,\, \forall \varepsilon>0.
\end{equation}
Note that $m_{\varepsilon}\geq 1$, $\forall\varepsilon>0$, the function $\varepsilon\mapsto m_{\varepsilon}$, $(0,\infty)\rightarrow [1,\infty)$, is decreasing and $m_{\varepsilon}\to \infty$, as $\varepsilon\to 0^+$.

\begin{thm}\label{Jaf detail}
Let $A=(A_{s,t})_{\Lambda\times\Lambda}:\ell^2(\Lambda)\to\ell^2(\Lambda)$ be an invertible matrix on $\ell^2(\Lambda)$ with inverse $(A^{-1}_{s,t})_{\Lambda\times\Lambda}=A^{-1}\in \mathcal{L}(\ell^2(\Lambda))$. Assume that $A\in\mathcal{E}_\gamma,$ that is, there exists $C_\gamma\geq 1$ so that
$$|A_{s,t}|\leq C_\gamma e^{-\gamma d(s,t)}\quad\mbox{for all}\quad s,t\in \Lambda.$$
Then there exist constants $\gamma_1\in(0,\gamma)$ and $C_{A,\gamma_1}>0$ such that
\begin{equation}\label{j ocena}
|A_{s,t}^{-1}|\leq C_{A,\gamma_1} e^{-\gamma_1 d(s,t)}\quad\mbox{for all}\quad s,t\in \Lambda.
\end{equation}
Furthermore, \eqref{j ocena} holds true with
\begin{equation}\label{form}
\gamma_1 = \min\left\{\delta, \frac{(\gamma'-\delta)\ln(1/r)}{\ln \left(\tilde{C} C^2_{\gamma}r^{-1}m^2_{(\gamma-\gamma')/2}\right)}\right\},\quad C_{A,\gamma_1}= \frac{2C_{\gamma}m_{\gamma-\gamma_1}}{(1-r)\|A\|^2}
\end{equation}
where $\gamma',\delta\in (0,\gamma)$, $0<\delta<\gamma'<\gamma$ are arbitrary,
$r=\|\operatorname{Id}-\|A\|^{-2}AA^*\|$ and $\tilde{C}=1+\|A\|^{-2}$.
\end{thm}

\begin{proof} The first part concerning the existence of $\gamma_1\in(0,\gamma)$ and $C_{A,\gamma_1}>0$ for which \eqref{j ocena} holds true is the same as in \cite{sjaff} and \cite{PS}. We add several new details to deduce the validity of \eqref{j ocena} with the constants given by \eqref{form}.\\
\indent Let $\tilde{A}=(\tilde{A}_{s,t})_{\Lambda\times\Lambda}=AA^*/\|A\|^2$, $R=(R_{s,t})_{\Lambda\times\Lambda}=\operatorname{Id}-\tilde{A}$ and denote $r=\|R\|$. As $\|AA^*\|=\|A\|^2$, we infer $r<1$ (this is the same constant $r$ as given in the statement of the theorem). Clearly $\tilde{A}^{-1}=(\tilde{A}^{-1}_{s,t})_{\Lambda\times\Lambda}=\operatorname{Id}+\sum_{n=1}^{\infty}R^n$. Let $\gamma',\delta\in (0,\gamma)$, $0<\delta<\gamma'<\gamma$, be arbitrary but fixed and denote $\gamma''=(\gamma'+\gamma)/2$. A straightforward computation yields
\begin{equation*}
|\tilde{A}_{s,t}|\leq C_{\gamma}^2\|A\|^{-2}m_{\gamma-\gamma''}e^{-\gamma''d(s,t)},\,\, \mbox{for all}\,\, s,t\in\Lambda.
\end{equation*}
As $m_{\gamma-\gamma''}\geq1$, the above implies $|R_{s,t}|\leq \tilde{C}C_{\gamma}^2m_{\gamma-\gamma''}e^{-\gamma''d(s,t)}$ with $\tilde{C}$ as in the statement of the theorem. Denoting $(R^n_{s,t})_{\Lambda\times\Lambda}=R^n$, $n\in\mathbb{Z}_+$, we readily deduce $|R^n_{s,t}|\leq r^n$, $s,t\in\Lambda$, $n\in\mathbb{Z}_+$. Furthermore, in the same way as in the proof of \cite[Proposition 2]{sjaff}, the above estimate for $R_{s,t}$ implies (notice that $m_{\gamma-\gamma''}=m_{\gamma''-\gamma'}=m_{(\gamma-\gamma')/2}$)
\begin{equation*}
|R^n_{s,t}|\leq (\tilde{C} C^2_{\gamma}m_{\gamma-\gamma''})^{n} m_{\gamma''-\gamma'}^{n-1}e^{-\gamma'd(s,t)}
\leq (\tilde{C} C^2_{\gamma}m^2_{(\gamma-\gamma')/2})^{n} e^{-\gamma'd(s,t)},
\end{equation*}
for all $ s,t\in\Lambda,\, n\in\mathbb{Z}_+.$ Consequently,
\begin{equation*}
|\tilde{A}^{-1}_{s,t}|\leq \sum_{n=0}^\infty \min\left\{r^n, (\tilde{C} C^2_{\gamma}m^2_{(\gamma-\gamma')/2})^{n} e^{-\gamma'd(s,t)}\right\},\,\, \mbox{for all}\,\, s,t\in\Lambda.
\end{equation*}
For each $s,t\in\Lambda$, let $n_0=n_0(s,t)\in\mathbb{N}$ be the largest $n\in\mathbb{N}$ such that
\begin{align*}
\left(\tilde{C} C^2_{\gamma}r^{-1}m^2_{(\gamma-\gamma')/2}\right)^n&\leq  e^{(\gamma'-\delta) d(s,t)};
\end{align*}
since the left hand side increases with $n$ the above holds for all $n\leq n_0$, $n\in\mathbb{N}$, and $n_0$ is given by
$$
n_0=\left \lfloor \frac{(\gamma'-\delta) d(s,t)}{\ln \left(\tilde{C} C^2_{\gamma}r^{-1}m^2_{(\gamma-\gamma')/2}\right)}\right\rfloor.
$$
Consequently,
\begin{equation}\label{j eksp}
r^{n_0+1}\leq e^{E(\gamma',\delta) \ln r d(s,t)},\,\,\, \mbox{with}\,\, E(\gamma',\delta)=\frac{(\gamma'-\delta)}{\ln \left(\tilde{C} C^2_{\gamma}r^{-1}m^2_{(\gamma-\gamma')/2}\right)}.
\end{equation}
We infer
\begin{eqnarray*}
|\tilde{A}^{-1}_{s,t}|&\leq& e^{-\delta d(s,t)}\sum_{n=0}^{n_0}(\tilde{C} C^2_{\gamma}m^2_{(\gamma-\gamma')/2})^{n}e^{-(\gamma'-\delta)d(s,t)} +\sum_{n=n_0+1}^{\infty}r^n\\
&\leq&e^{-\delta d(s,t)}\sum_{n=0}^{n_0}r^n +e^{E(\gamma',\delta) \ln r d(s,t)}\sum_{n=0}^{\infty}r^n\\
&\leq &\frac{e^{-\delta d(s,t)}}{1-r} + \frac{e^{-E(\gamma',\delta) \ln (1/r) d(s,t)}}{1-r}.
\end{eqnarray*}
As $A^{-1}=A^*\tilde{A}^{-1}/\|A\|^2$, we conclude the validity of \eqref{j ocena} with $\gamma_1$ and $C_{A,\gamma_1}$ given by \eqref{form}.
\end{proof}

\begin{rem} As we mentioned before, if $\gamma'\to \gamma$ then $m_{\gamma-\gamma'}\to \infty$. Therefore, the constant $\gamma_1$ can be arbitrary close to $0$ when $\gamma'$ approaches $\gamma$ and $\delta$ approaches $0$.
\end{rem}


\section{Estimates depending on the growth rate of constants}

\begin{rem}\label{rem 1}
\begin{enumerate}
\item Example \ref{ex} shows that, for arbitrary $k\in \mathbb{Z}_+$, there is a matrix $A\in \bigcap_{\gamma>0}\mathcal{E}_\gamma$ such that $A^{-1}\in \mathcal{E}_{\gamma_1}$ only for $\gamma_1\leq 1/k$. So, at best, one can claim that, even if $A\in \bigcap_{\gamma>0}\mathcal{E}_\gamma,$ the inverse $A^{-1}\in \bigcup_{\gamma>0}\mathcal{E}_\gamma.$

\item The super-exponential decay can not give the spectral invariance. Let $\beta>1$ and $|A_{s,t}|\leq C_p e^{-p|s-t|^\beta}$. The infinite matrix in Example \ref{ex} satisfies this condition but its inverse has only exponential decay.
\end{enumerate}
\end{rem}

Guided by that, in this section, our goal is to provide estimates on the rate of decay of the inverse when $A\in \bigcap_{p\in\mathbb{N}}\mathcal{E}_p$ with additional assumptions on the constants $C_p$ in the definition of $\mathcal{E}_p$ (cf. \eqref{*}).

Assume that the matrix type operator $A=(A_{s,t})_{\Lambda\times\Lambda}:\ell^2(\Lambda)\rightarrow \ell^2(\Lambda)$ satisfies
\begin{equation}\label{j1}
(\forall p\geq1)(\exists C_p\geq1)\,\,|A_{s,t}|\leq C_p e^{-p d(s,t)},\;\mbox{for all} \;s,t \in\Lambda.
\end{equation}
Let $\varphi$ be a strictly increasing function $\varphi:[1,\infty)\to [1,\infty)$ which satisfies the following condition:
\begin{equation}\label{imo}
\limsup_{p\to \infty}\frac{C_p}{e^{p\varphi(p)}}=K\in[1,\infty).
\end{equation}

As we do not want to consider band limited matrices, the constants $C_p$, $p\in [1,\infty)$, should not be bounded from above (otherwise, $A_{s,t}=0$ for all $s,t\in\Lambda$, $s\neq t$). Let $\tilde{K}>K$. Condition \eqref{imo} implies that $C_p\leq \tilde{K} e^{p\varphi(p)}$ for all $p \in [1,\infty)$ except for finitely many of them. We discuss the following cases.

\begin{enumerate}
\item If the function $\varphi$ is bounded, i.e. if there exists $k_0\geq 1$ such that $\varphi(p)\leq k_0$ for all $p$, then $C_p\leq \tilde K e^{k_0 p}$, for all but finitely many $p\in [1,\infty)$. So, for $d(s,t)> k_0$ one obtains $A_{s,t}=0$ (i.e., the band limited case).

\item If the function $\varphi$ is unbounded, i.e. $\lim_{p\to \infty}\varphi(p)= \infty$, then
$$|A_{s,t}|\leq K_1e^{p(\varphi(p)- d(s,t))}, \quad \mbox{for all}\,\, s,t\in\Lambda,\, p\geq 1,$$
with $K_1=\sup_{p\in[1,\infty)}C_pe^{-p\varphi(p)}\geq 1.$
\end{enumerate}
\begin{ex}
 Consider the function $\varphi(p)=c_0 p^h$, for some $c_0,h>0$. Then, $\lim_{p\to \infty}\varphi(p)= \infty$. As the function in the exponent has minimum at $\tilde{p}=(\frac{d(s,t)}{c_0(1+h)})^{1/h}$, one obtains
\begin{align*}
|A_{s,t}| & \leq K_1e^{c_0}e^{-d(s,t)}, \quad \mbox{for} \quad d(s,t)\leq c_0(1+h),\\
|A_{s,t}| & \leq K_1e^{-\tilde c_0 d(s,t)^{1+1/h}}, \quad \mbox{for} \quad d(s,t) >c_0(1+h),
\end{align*}
where $\tilde{c}_0=hc_0^{-1/h}(1+h)^{-1-1/h}$. So, this case includes matrices that are not necessarily banded, i.e. they can have all non-zero entries.
\end{ex}

For the next theorem we assume $\varphi:[1,\infty)\to [1,\infty)$ has the following properties:
\begin{itemize}
\item[$(a)$] $\lim_{p\to \infty} \varphi(p)=\infty$ and $\varphi(1)=1$,
\item[$(b)$] there exists $a>1$ such that $\varphi(\xi p)\leq \xi^{a-1}\varphi(p),\;p,\xi\in [1,\infty)$.
\end{itemize}
The fact that $\varphi$ is strictly increasing together with $(a)$ and $(b)$ implies that $\varphi$ is continuous, see Lemma \ref{app}, and consequently bijective.
\begin{ex}
The functions $\varphi(p)=  p^\alpha$, $\alpha>0$, and $\varphi(p)=\ln(p+e-1)$, $\varphi: [1,\infty)\to[1,\infty)$, provide examples which satisfy all of these conditions. Products of such functions also satisfy the above conditions.
\end{ex}

\begin{thm}\label{sjt}
Assume the matrix type operator $A=(A_{s,t})_{\Lambda\times\Lambda}:\ell^2(\Lambda)\rightarrow \ell^2(\Lambda)$ is invertible with $(A^{-1}_{s,t})_{\Lambda\times\Lambda}=A^{-1}$ being its inverse. If $A$ satisfies \eqref{j1} and \eqref{imo}, with $\varphi:[1,\infty)\rightarrow [1,\infty)$ a strictly increasing function which satisfies \eqref{imo}, $(a)$ and $(b)$, then there exist $C_A,b>0$ such that
	\begin{equation}\label{j2}
	|A^{-1}_{s,t}|\leq C_A e^{-b d(s,t)}, \quad \mbox{for all\;} s,t\in\Lambda.
	\end{equation}
Furthermore, \eqref{j2} holds true with
\begin{equation}\label{vckltr135}
b=\frac{\ln(1/r)}{\ln(\tilde{C}K_1^2 m_1^2r^{-1})+2\cdot 4^a}\,\,\, \mbox{and}\,\,\, C_A=\frac{2C_2m_1}{(1-r)\|A\|^2},
\end{equation}
where $r=\|\operatorname{Id}-\|A\|^{-2}AA^*\|$, $\tilde{C}=1+\|A\|^{-2}$, $K_1=\sup_{p\in[1,\infty)}C_pe^{-p\varphi(p)}\geq1$.
\end{thm}

\begin{proof} The fact that \eqref{j2} holds follows from \cite[Proposition 2]{sjaff}; we need to prove the validity of \eqref{j2} with the constants given in \eqref{vckltr135}. For this purpose, we keep the same notations for $\tilde{A}$ and $R$ given in the proof of Theorem \ref{Jaf detail}.\\
\indent One easily deduces that
\begin{equation*}
|\tilde{A}_{s,t}|\leq C_{p+1}^2\|A\|^{-2}m_1e^{-pd(s,t)},\,\, \mbox{for all}\,\, s,t\in\Lambda,\, p\geq 1,
\end{equation*}
which, in turn, yields
\begin{equation}\label{vstekp139}
|R_{s,t}|\leq \tilde{C}C_{p+1}^2m_1e^{-pd(s,t)},\,\, \mbox{for all}\,\, s,t\in\Lambda,\, p\geq 1,
\end{equation}
with $\tilde{C}$ as in the statement of the theorem. To estimate $R^n_{s,t}$, for the moment, we denote $\gamma_p=p+2$ and $\gamma'_p=p+1$. For $n\geq 2$, we infer
\begin{align*}
|R^n_{s,t}|
& \leq \sum_{s_1\in\Lambda}\dots \sum_{s_{n-1}\in\Lambda}
|R_{s,s_1}||R_{s_1,s_2}|\cdots|R_{s_{n-1},t}|\\
&\leq (\tilde{C}C^2_{p+3}m_1)^n\sum_{s_1\in\Lambda}\dots \sum_{s_{n-1}\in\Lambda} e^{-\gamma'_p d(s,s_1)}\cdots e^{-\gamma'_p d(s_{n-1},t)}\cdot \\
&\hspace{8.5cm} \cdot e^{-d(s,s_1)}\cdots e^{-d(s_{n-1},t)}\\
&\leq (\tilde{C}C^2_{p+3}m_1)^ne^{-\gamma'_p d(s,t)}\sum_{s_1\in\Lambda} \dots \sum_{s_{n-1}\in\Lambda}
e^{-d(s,s_1)}\cdots e^{-d(s_{n-2},s_{n-1})}\\
&\leq (\tilde{C}C^2_{p+3}m_1)^n m_1^{n-1}e^{-\gamma'_p d(s,t)}.
\end{align*}
Denoting $\varphi_1(p)=p\varphi(p)$, $p\geq 1$, we conclude
\begin{equation*}
|R^n_{s,t}| \leq (\tilde{C}C^2_{p+3}m^2_1)^ne^{-(p+1) d(s,t)}\leq (\tilde{C}K^2_1m^2_1)^ne^{2n\varphi_1(p+3)}e^{-(p+1) d(s,t)},
\end{equation*}
for all $s,t\in\Lambda,\, n,p\in \mathbb{Z}_+$ (for $n=1$ this trivially holds because of \eqref{vstekp139}). Let $\varepsilon>0$ be arbitrary but fixed. Specialising the above for $p=\varphi^{-1}(n^{\varepsilon})\geq 1$, we deduce
\begin{equation*}
|R^n_{s,t}| \leq (\tilde{C}K^2_1m^2_1)^ne^{2n\varphi_1(\varphi^{-1}(n^{\varepsilon})+3)}e^{-\varphi^{-1}(n^{\varepsilon})d(s,t)} e^{-d(s,t)},
\end{equation*}
for all $s,t\in\Lambda,\, n\in \mathbb{Z}_+,$ which, in turn, gives
\begin{equation}\label{vktspl137}
|\tilde{A}^{-1}_{s,t}|\leq \delta_{s,t}+
\sum_{n=1}^{\infty} \min\left\{r^n, (\tilde{C}K^2_1m^2_1)^ne^{2n\varphi_1(\varphi^{-1}(n^{\varepsilon})+3)}e^{-\varphi^{-1}(n^{\varepsilon})d(s,t)} e^{-d(s,t)}\right\},
\end{equation}
where $\delta_{s,t}$ stands for the Kronecker delta. Given $s,t\in\Lambda$, let $n_0=n_0(s,t,\varepsilon)\in\mathbb{N}$ be the largest $n\in\mathbb{N}$ such that
\begin{equation*}
n^{1+\varepsilon}(\ln(\tilde{C}K_1^2m_1^2r^{-1})+2\cdot 4^a)\leq d(s,t).
\end{equation*}
Since the left-hand side increases with $n$, the above holds true for all $n\leq n_0$; clearly
\begin{equation*}
n_0=\left\lfloor D(\varepsilon)d(s,t)^{1/(1+\varepsilon)}\right\rfloor,\,\,\, \mbox{with}\,\,\, D(\varepsilon)=(\ln(\tilde{C}K_1^2 m_1^2r^{-1})+2\cdot 4^a)^{-1/(1+\varepsilon)}.
\end{equation*}
To estimate $\tilde{A}^{-1}_{s,t}$, we split the series \eqref{vktspl137} at $n_0$ and write\footnote{In the following, we employ the principle of vacuous (empty) sum, i.e. $\sum_{j=1}^0r_j=0$; this can happen when $n_0=0$}
\begin{align*}
|\tilde{A}^{-1}_{s,t}|\leq \delta_{s,t} & +\sum_{n=1}^{n_0} (\tilde{C}K^2_1m^2_1)^ne^{2n\varphi_1(\varphi^{-1}(n^{\varepsilon})+3)}e^{-\varphi^{-1}(n^{\varepsilon})d(s,t)} e^{-d(s,t)}
 + \sum_{n=n_0+1}^{\infty}r^n \\
 & =\delta_{s,t}+S_1+S_2.
\end{align*}
As $D(\varepsilon)\geq D(0)=D=(\ln(\tilde{C}K_1^2 m_1^2r^{-1})+2\cdot 4^a)^{-1}$, we infer
\begin{equation*}
S_2=\frac{r^{n_0+1}}{1-r}\leq \frac{e^{-D\ln(1/r)d(s,t)^{1/(1+\varepsilon)}}}{1-r}.
\end{equation*}
To estimate $S_1$, we employ the monotonicity of $\varphi$, $\varphi(n^{\varepsilon})\geq1$ and the property $(b)$ to infer (for $1\leq n\leq n_0$)
\begin{align*}
&(\tilde{C}K^2_1m^2_1)^ne^{2n\varphi_1(\varphi^{-1}(n^{\varepsilon})+3)}=\\
&\hspace{2cm}  = r^n\exp\left(n\ln(\tilde{C}K_1^2m_1^2r^{-1})+2n(\varphi^{-1}(n^{\varepsilon})+3)\varphi(\varphi^{-1}(n^{\varepsilon})+3)\right)\\
&\hspace{2cm}  \leq r^n\exp\left(n\ln(\tilde{C}K_1^2m_1^2r^{-1})+8n\varphi^{-1}(n^{\varepsilon})\cdot 4^{a-1}\varphi(\varphi^{-1}(n^{\varepsilon}))\right)\\
&\hspace{2cm}  \leq r^n\exp\left(\varphi^{-1}(n^{\varepsilon})n^{1+\varepsilon}\left(\ln(\tilde{C}K_1^2m_1^2r^{-1})+2\cdot4^a\right)\right)\\
&\hspace{2cm}  \leq r^ne^{\varphi^{-1}(n^{\varepsilon})d(s,t)}.
\end{align*}
Consequently,
\begin{equation*}
\delta_{s,t}+S_1\leq e^{-d(s,t)}\sum_{n=0}^{n_0}r^n\leq \frac{e^{-d(s,t)}}{1-r}.
\end{equation*}
Thus,
\begin{equation*}
|\tilde{A}^{-1}_{s,t}|\leq \frac{e^{-d(s,t)}}{1-r} +\frac{e^{-D\ln(1/r)d(s,t)^{1/(1+\varepsilon)}}}{1-r},
\end{equation*}
and, as $\varepsilon>0$ is arbitrary, we conclude (notice that $D\ln(1/r)\leq1$)
\begin{equation*}
|\tilde{A}^{-1}_{s,t}|\leq \frac{2e^{-D\ln(1/r)d(s,t)}}{1-r},\,\,\, \mbox{for all}\,\, s,t\in\Lambda.
\end{equation*}
As $A^{-1}=A^*\tilde{A}^{-1}/\|A\|^2$, we conclude the validity of \eqref{j2} with $b$ and $C_A$ given by \eqref{vckltr135}.
\end{proof}

\section*{Appendix}
\begin{lemma}\label{app}
Let $\varphi:[1,\infty)\rightarrow[1,\infty)$ be strictly increasing function which satisfies the following properties:
\begin{itemize}
\item[$(a)$] $\lim_{p\to \infty} \varphi(p)=\infty$ and $\varphi(1)=1$,
\item[$(b)$] there exists $a>1$ such that $\varphi(\xi p)\leq \xi^{a-1}\varphi(p),\;p,\xi\in [1,\infty)$.
\end{itemize}
Then $\varphi$ is continuous. 
\end{lemma}

\begin{proof}
Let $t_0\in(1,\infty)$. Denote
$$
A=\sup_{t\in[1,t_0)}\varphi(t),\,\,\, B=\inf_{t\in(t_0,\infty)}\varphi(t).
$$
Since $\varphi$ is strictly increasing, one obtains that $1\leq A\leq \varphi(t_0)$ and $B\geq \varphi(t_0)$. We claim $A=\varphi(t_0)=B$. Assume that $A<\varphi(t_0)$. Then there exists $\varepsilon>0$ such that $A<\varphi(t_0)-\varepsilon$. For every $t\in[1,t_0),$ employing $(b),$ one deduces
$$
\varphi(t_0)=\varphi\left(t\cdot\frac{t_0}{t}\right)\leq \left(\frac{t_0}{t}\right)^{a-1}\varphi(t) <\left(\frac{t_0}{t}\right)^{a-1}(\varphi(t_0)-\varepsilon).
$$
When $t\rightarrow t_0^-,$ it leads to a contradiction. Similarly, if one assumes $B>\varphi(t_0)$, there exists $\varepsilon>0$ such that $B>\varphi(t_0)+\varepsilon$. For $t\in(t_0,\infty)$, we infer
$$
\varphi(t_0)+\varepsilon<\varphi(t)=\varphi\left(t_0\cdot\frac{t}{t_0}\right)\leq \left(\frac{t}{t_0}\right)^{a-1}\varphi(t_0),
$$
which is a contradiction if one lets $t\rightarrow t_0^+$.Thus $A=\varphi(t_0)=B$. As $\varphi$ is strictly increasing, the latter immediately implies that $\varphi$ is continuous at $t_0\in(1,\infty)$. The continuity at $t_0=1$ can be proved in the similar manner. 
\end{proof}


\section*{Acknowledgement}

The work of S. Pilipovi\'c is supported by the Ministry of Education, Science and Technological Development of the Republic of Serbia [research project 174024].

The work of M. \v Zigi\' c is partially supported by the Ministry of Education, Science and Technological Development of the Republic of Serbia [research project 174024] and the Provincial Secretariat for Science of Vojvodina [research project APV 142-451-2102/2019].

The work of B. Prangoski was partially supported by the bilateral project "Microlocal analysis and applications" funded by the Macedonian and Serbian academies of sciences and arts.

\end{document}